\documentclass[leqno,a4paper]{article}
\usepackage{amssymb,amsmath,amsthm,verbatim}

\title{
\textbf{The Classification Problem for}
$2$\textbf{-Forms in Four Variables}
}
\author{\textsc{J. Mu\~{n}oz Masqu\'e}
\and \textsc{L. M. Pozo Coronado}}
\date{}

\newtheorem{theorem}{Theorem}[section]
\newtheorem{proposition}[theorem]{Proposition}
\newtheorem{lemma}[theorem]{Lemma}
\newtheorem{corollary}[theorem]{Corollary}

\theoremstyle{remark}
\newtheorem{remark}[theorem]{Remark}

\newtheorem{example}[theorem]{Example}

\begin{document}

\maketitle

\begin{abstract}
\noindent The notion of \textit{type} of a differential $2$-form
in four variables is introduced and for $2$-forms of type $<4$,
local normal models are given. If the type of a $2$-form $\Omega $
is $4$, then the equivalence under diffeomorphisms of $\Omega $
is reduced to the equivalence of a symplectic linear frame
functorially attached to $\Omega $.
As the equivalence problem for linear parallelisms is known,
the present work solves generically the equivalence problem under
diffeomorphisms of germs of $2$-forms in $4$ variables.
\end{abstract}
\medskip
\noindent \textit{Mathematics Subject Classification 2010:\/}
Primary: 58A10; Secondary: 53A55, 58A20.
\medskip

\noindent \textit{Key words and phrases:\/} Class of a $2$-form,
differential $2$-form, jet bundle, linear frame, local model,
type of a $2$-form.
\section{Introduction and preliminaries\label{intro/prelim}}
\subsection{The type of a $2$-form}
\begin{lemma}
[$1$-form attached to a $2$-form]\label{lemma_omega}
Let $M$ be a smooth manifold of dimension $4$, and let
$p\colon O_M\to M$ be the dense open subbundle in
$\wedge ^2T^\ast M$ of $2$-covectors of maximal rank;
i.e.,
$O_M=\{ w\in\wedge ^2T^\ast M:\operatorname{rk}w=4\} $.
For every $2$-form $\Omega $ taking values in $O_M$
there exists a unique $1$-form $\omega _\Omega $ such that,
\begin{equation}
\omega _\Omega \wedge \Omega=d\Omega .
\label{omega/m}
\end{equation}
The assignment $\Omega \mapsto \omega _\Omega $
has a functorial character; i.e., for every
$\phi \in \operatorname{Diff}M$ we have
$\phi ^\ast \omega _\Omega =\omega _{\phi ^\ast\Omega }$.
\end{lemma}
\begin{proof}
In fact, we have $\omega _\Omega =2\,i_X\Omega $,
where $X$ is the vector field defined by the equation
$i_X\left( \Omega \wedge \Omega \right) =d\Omega $.
\end{proof}
The $1$-form $\omega _\Omega $ is the opposite to
the `curvature covariant vector' introduced in \cite{Lee}.

Let $M$ be a manifold of dimension $4$.
Following \cite[VI, 1.3]{Godbillon} or
\cite[Appendix 4, 3.5]{LibermannMarle}, we define the class
of a differential form $\omega $ at a point $x\in M$
as the codimension of the subspace of the tangent vectors
$X\in T_xM$ such that
$i_X\omega =0$, $i_X(d\omega )=0$. Let $\mathcal{G}$
be the set of germs of differential $2$-forms of rank $4$
at a point $x_0\in M$ defined as follows:
$\Omega $ belongs to $\mathcal{G}$ if and only if
the germ of differential $1$-form $\omega _\Omega $
determined by the equation \eqref{omega/m} is not singular
at $x_0$, i.e., $(\omega _\Omega )_{x_0}\neq 0$,
and of constant class on a neighbourhood of $x_0$.
A germ $\Omega \in \mathcal{G}$ is said to be
of \textit{type} $t\in\{ 1,2,3,4\} $ if $\omega _\Omega $
is of class $t$. The germs of type $0$ are the closed
$2$-forms of rank $4$.
\subsection{Summary of contents}
The main goal of this article is to give a procedure
to decide when two germs of $2$-forms in four variables
are equivalent under diffeomorphisms.

In section \ref{types0/1/2} the models of germs
of $2$-forms of types $1$ and $2$ are given. In fact,
in Theorems \ref{type1} and \ref{type2} it is proved
that there exists a unique model for each of these two types.
This result is analogous to the Darboux' Theorem,
which applies to germs of $2$-forms of type $0$.

In section \ref{type3}, a model for germs of $2$-forms
of type $3$ is given. Furthermore, a finer classification
of germs of $2$-forms of type $3$ is given,
in terms of the class of a $1$-form $\varphi _\Omega $
associated to the original $2$-form $\Omega $.

In section \ref{inv1}, the notion of functions
on $2$-forms invariant under diffeomorphisms,
is introduced.

In section \ref{type4}, generic germs of $2$-forms
of type $4$ are studied.

The conditions for a general germ of $2$-form $\Omega $
to have an $\omega _\Omega $ in normal form, i.e.,
$\omega _\Omega =(1+x^1)dx^2+x^3dx^4$, are given. Finally,
a linear frame attached to a germ or $2$-form of type $4$
is given and a result is proved (Theorem \ref{TheBasis}),
stating that the equivalence of germs of $2$-forms
under diffeomorphims is reduced to the equivalence
under diffeomorphisms of the corresponding linear frames.
As the equivalence problem for linear parallelisms
is solved, this result thus fulfills our goal
in the generic case.
\section{Forms of types $1$ and $2$\label{types0/1/2}}
\begin{theorem}\label{type1}\emph{(cf.\ \cite[\S 2]{Debever})}
For every $\Omega \in\mathcal{G}$ of type $1$, there exists
a coordinate system
$(x^i)_{i=1}^4 $
centred at the origin such that,
\begin{equation}
\Omega=\exp (x^3)dx^1\wedge dx^2+dx^3\wedge dx^4.
\label{type1eq}
\end{equation}
\end{theorem}
\begin{proof}
Let $Z$ be the vector field determined by
$i_Z\Omega=\omega _\Omega =\omega $.
Hence $i_Z\omega =0$ and then, $i_Zd\Omega =0$ as follows
from the formula \eqref{omega/m}. Moreover, by virtue
of the hypothesis, there exists a smooth function such that
$\omega =df$. Hence $L_Z\Omega =0$. As $Z$ is not singular
at the origin, there exists a smooth function $g$ such that
$Z(g)=1$. Furthermore, we can assume that $f$ and $g$ vanish
at the origin. Let $Y$ be the vector field defined by
$i_Y\Omega =dg$. We claim that $Y,Z$ are linearly
independent. In fact, from the very definition
of $f$ and $Y$ we obtain $df(Z)=dg(Y)=0$, and also,
\begin{equation}
df(Y)=i_Z\Omega (Y)=-i_Y\Omega (Z)=-dg(Z)=-1.
\label{df(Y)}
\end{equation}
Hence $(df\wedge dg)(Z,Y)=1$. Next, we prove that
$\exp (-f)$ is an integrating factor for
$\tilde{\Omega }=\Omega +df\wedge dg$. As
$d(\exp (-f)\tilde{\Omega })$ is a form of degree three
in four variables, it suffices to state
$i_Yd(\exp (-f)\tilde{\Omega })=0$ and
$i_Zd(\exp (-f)\tilde{\Omega })=0$.
We start with the latter equation:
$i_Zd(\exp (-f)\tilde{\Omega })
=i_Z(-\exp (-f)df\wedge\Omega +\exp (-f)d\Omega )=0$.
As for the former equation, taking account
of the equation \eqref{df(Y)} and the identity
$i_Y\tilde{\Omega }=0$, we have
\begin{align*}
i_Yd\bigl( \exp (-f)\tilde{\Omega }\bigr)
& =i_Y\bigl(-\exp (-f)df\wedge
\tilde{\Omega }+\exp (-f)d\Omega\bigr)\\
& =\exp (-f)\bigl( \tilde{\Omega }+i_Yd\Omega \bigr) .
\end{align*}
Hence it suffices to prove $\tilde{\Omega }+i_Yd\Omega =0$.
To do this, we first remark that the equation \eqref{omega/m}
can be rewritten as $d\Omega =i_Z\Omega \wedge\Omega $; hence
\[
i_Yd\Omega
=\bigl( i_Yi_Z\Omega \bigr) \Omega -i_Z\Omega \wedge i_Y\Omega
=\Omega (Z,Y)\Omega -df\wedge dg=-\tilde{\Omega }.
\]
According to Darboux's theorem, there exist functions $x^1,x^2$,
vanishing at the origin, such that
$\exp (-f)\tilde{\Omega }=dx^1\wedge dx^2$. Setting $x^3=f$
and $x^4=-g$, we obtain \eqref{type1eq}. As $\Omega $
is of rank $4$ we have $\Omega \wedge \Omega \neq 0$,
and the system $(x^i)_{i=1}^4$ is functionally independent.
\end{proof}
\begin{lemma}
\label{type2a}
For every $\Omega \in \mathcal{G}$ of type $2$, there exist
coordinates $(x^i)_{i=1}^4$ centred at the origin
and a smooth function $u=u(x^1,x^2,x^3)$, vanishing
at the origin, such that
\begin{equation}
\bigl( (u_{x^1})^2+(u_{x^2})^2\bigr) (0,0,0)>0,
\label{positive}
\end{equation}
\begin{equation}
\Omega
=\exp \left( (1+u)x^3\right) dx^1\wedge dx^2+dx^3\wedge dx^4.
\label{type2aeq}
\end{equation}
\end{lemma}
\begin{proof}
Let $Z$ be the vector field determined by
$i_Z\Omega =\omega _\Omega $. As $\Omega $ is of type $2$,
there exist smooth functions vanishing at the origin
such that $i_Z\Omega =(1+y^1)dy^3$. Hence $Zy^3=0$. As $Z$
is not singular, there exists $y^4$ vanishing at the origin,
such that $Zy^4=-(1+y^1)$. Let $Y$ be the vector field
defined by $i_Y\Omega =dy^4$. Hence $Yy^4=0$. We also have
$(1+y^1)(Yy^3)=i_Yi_Z\Omega =-i_Zi_Y\Omega =1+y^1 $,
so that $Yy^3=1$. Moreover, we claim that $Zy^1=0$. In fact,
this equation is equivalent to $L_Zi_Z\Omega =0$, since
$L_Zi_Z\Omega =(Zy^1)dy^3$. In order to prove that
$L_Zi_Z\Omega $ vanishes, we proceed as follows:
As $i_Zd\Omega =0$, we have $L_Zd\Omega=0$, or equivalently
$L_Z(i_Z\Omega \wedge \Omega )=0$;
that is, $L_Zi_Z\Omega \wedge \Omega
+i_Z\Omega \wedge L_Z\Omega =0 $, and the second term
on the left hand side vanishes, as
$i_Z\Omega \wedge L_Z\Omega =i_Z\Omega \wedge di_Z\Omega $,
and $i_Z\Omega $ is of class $2$. Hence
$L_Z(i_Z\Omega )\wedge \Omega =0$, and therefore
$L_Zi_Z\Omega =0$. We conclude that $y^1,y^3$ are first
integrals of $Z$ but $y^4$ is not. Consequently,
we can complete $y^1,y^3,y^4$ up to a system
of coordinates $(y^1,\dotsc,y^4)$ such that $Zy^2=0$.
Moreover, $Y,Z$ are linearly independent as
$(dy^3\wedge dy^4)(Z,Y)=1+y^1\neq 0$.

Similarly to the case of forms of type $1$, we look for
integrating factors for $\tilde{\Omega }=\Omega-dy^3\wedge dy^4$.
For every smooth function $\rho $ we have
\[
i_Zd\bigl( \rho\tilde{\Omega }\bigr)
=Z(\rho)\tilde{\Omega },\quad
i_Yd\bigl( \rho\tilde{\Omega }\bigr)
=\bigl( Y\rho+\rho(1+y^1)\bigr) \tilde{\Omega }.
\]
In fact, on one hand we have
$i_Zd(\rho\tilde{\Omega })=Z(\rho)\tilde{\Omega }
+\rho i_Zd\Omega =Z(\rho)\tilde{\Omega }$, and on the other,
$i_Yd(\rho\tilde{\Omega })
=Y(\rho)\tilde{\Omega }+\rho i_Yd\Omega $
and
\begin{align*}
i_Yd\Omega &
=i_Y(i_Z\Omega\wedge\Omega )
=\Omega (Z,Y)\Omega-i_Z\Omega\wedge i_Y\Omega \\
& =(1+y^1)\Omega -(1+x^1)dy^3\wedge dy^4\\
& =(1+y^1)\tilde{\Omega }.
\end{align*}
We thus conclude that $\rho$ is an integrating factor
for $\tilde{\Omega }$ if and only if $Z\rho =0$ and
$Y\rho +\rho (1+y^1)=0$. The former equation simply
means that $\rho =\rho (y^1,y^2,y^3)$.
As for the latter equation, it is a first-order
non-singular linear PDE in the unknown $\rho $
on the space $C^\infty (y^1,y^2,y^3)$. Hence the existence
of a positive integrating factor for $\tilde{\Omega }$
on such space, is guaranteed. According to Darboux's theorem,
we have $\rho \tilde{\Omega }=dt^1\wedge dt^2$; i.e.,
$\Omega =\exp (r)dt^1\wedge dt^2+dy^3\wedge dy^4$ with
$r=-\ln \rho $. Since $\Omega $ is of rank $4$, we conclude
that $t^1,t^2,y^3,y^4$ are functionally independent,
and taking the equations
$i_Y\tilde{\Omega }=i_Z\tilde{\Omega }=0$ into account,
we obtain $Yt^a=Zt^a=0$, $a=1,2$. On the coordinate system
$(t^1,t^2,y^3,y^4)$ we thus have
$Y=\partial/\partial y^3$ and $Z=(1+y^1)\partial/\partial y^4$,
with $y^1=f(t^1,t^2,y^3)$. Hence $r=r(t^1,t^2,y^3)$ and,
as a simple calculation shows, we have
$\omega _\Omega =i_Z\Omega =r_{y^3}dy^3$.
Since $\omega _\Omega (0)\neq 0$, we deduce that
$\lambda=r_{y^3}(0,0,0)\neq 0$.
Expanding $r$ up to the first order, we obtain
$r(t^1,t^2,y^3)=r_0(t^1,t^2)+u(t^1,t^2,y^3)y^3$,
with $\lambda=u(0,0,0)$. Making the change of variables
\[
\begin{array}
[c]{lll}
x^1=\int\exp \left( r_0(t^1,t^2)\right) dt^1, & x^2=t^2, &
x^i=y^i,\;i=3,4,
\end{array}
\]
we can assume $r_0\equiv0$ (hence $r=ux^3$), and substituting
$u/\lambda$ for $u$ and $\lambda x^3$ for $x^3$, we can also
assume $u(0,0,0)=1$. The condition \eqref{positive}
on the statement of the lemma follows from the fact that
$\omega _\Omega $ is of rank $2$. This completes the proof.
\end{proof}
\begin{theorem}
\label{type2}
For every $\Omega\in\mathcal{G}$ of type $2$, there exists
a coordinate system centred at the origin $(x^i)_{i=1}^4$,
such that
\begin{equation}
\Omega =\exp \left( (1+x^1)x^3\right) dx^1\wedge dx^2
+dx^3\wedge dx^4.
\label{type2eq}
\end{equation}
\end{theorem}
\begin{proof}
We start with the reduced equation for $\Omega $ given
in Lemma \ref{type2a}; namely,
 $\Omega =\exp (r)dx^1\wedge dx^2+dx^3\wedge dx^4$,
$r=(1+u)x^3$, and $u=u(x^1,x^2,x^3)$ satisfies
the inequation \eqref{positive}.

We look for a system of coordinates $(y^i)_{i=1}^4$
centred at the origin, such that
\begin{equation}
\Omega =\exp ((1+y^1)y^3)dy^1\wedge dy^2+dy^3\wedge dy^4.
\label{type2yeq}
\end{equation}
We set
\begin{equation}
\mathrm{(i)\;\,}y^1=r_{x^3}-1=u+x^3u_{x^3},
\qquad\mathrm{(ii)\;\,}
y^3=x^3.
\label{y1y3}
\end{equation}
Comparing the corresponding coefficients of $dx^i\wedge dx^j$,
$1\leq i<j\leq 4$, in \eqref{type2aeq} and in \eqref{type2yeq},
and taking \eqref{y1y3} into account, we obtain
\begin{align}
\exp \left( (1+y^1)y^3\right) \det \tfrac{\partial(y^1,y^2)}
{\partial(x^1,x^2)} & =\exp r,\label{eq12'}\\
\exp \left( (1+y^1)y^3\right) \det \tfrac{\partial(y^1,y^2)}
{\partial(x^1,x^3)}-\frac{\partial y^4}{\partial x^1} &
=0,\label{eq13'}\\
\exp \left( (1+y^1)y^3\right) \det \tfrac{\partial(y^1,y^2)}
{\partial(x^1,x^4)} & =0,\label{eq14'}\\
\exp \left( (1+y^1)y^3\right) \det \tfrac{\partial(y^1,y^2)}
{\partial(x^2,x^3)}-\tfrac{\partial y^4}{\partial x^2} &
=0,\label{eq23'}\\
\exp \left( (1+y^1)y^3\right) \det \tfrac{\partial(y^1,y^2)}
{\partial(x^2,x^4)} & =0, \label{eq24'}\\
\exp \left( (1+y^1)y^3\right) \det \tfrac{\partial(y^1,y^2)}
{\partial(x^3,x^4)}+\tfrac{\partial y^4}{\partial x^4} &
=1.
\label{eq34'}
\end{align}
From \eqref{y1y3}-(i) it follows that $y^1$ does not depend
on $x^4$; hence \eqref{eq14'} is equivalent to
$\partial y^1/\partial x^1\cdot\partial y^2/\partial x^4=0$,
\eqref{eq24'} is equivalent to
$\partial y^1/\partial x^2\cdot\partial y^2/\partial x^4=0$,
and \eqref{eq34'} is equivalent to
\begin{equation}
\exp ((1+y^1)y^3)\tfrac{\partial y^1}{\partial x^3}
\tfrac{\partial y^2}{\partial x^4}
+\tfrac{\partial y^4}{\partial x^4}=1.
\label{y4x4}
\end{equation}
If $\partial y^2/\partial x^4\neq 0$, then
$\partial y^1/\partial x^1=\partial y^1/\partial x^2=0$;
hence $y^1=y^1(x^3)$, which is absurd, as $y^1$ and $y^3$
are assumed to be independent. Accordingly,
$\partial y^2/\partial x^4=0$. From \eqref{y4x4} we thus deduce
$\partial y^4/\partial x^4=1$; hence
\begin{equation}
\mathrm{(i)\;\,}y^2=y^2(x^1,x^2,x^3),
\quad\mathrm{(ii)\;\,}y^4=x^4+C(x^1,x^2,x^3),
\label{y2y4}
\end{equation}
and the equations \eqref{eq14'}, \eqref{eq24'},
\eqref{eq34'} hold.
Hence the system \eqref{eq12'}--\eqref{eq34'}
reduces to the following:
\begin{align}
\det \tfrac{\partial(y^1,y^2)}{\partial(x^1,x^2)} &
=\exp (r-B),\label{x1x2}\\
\det \tfrac{\partial(y^1,y^2)}{\partial(x^1,x^3)} &
=\tfrac{\partial C}{\partial x^1}\exp (-B),\label{x1x3}\\
\det \tfrac{\partial(y^1,y^2)}{\partial(x^2,x^3)} &
=\tfrac{\partial C}{\partial x^2}\exp (-B),\label{x2x3}
\end{align}
with $B=(1+y^1)y^3$. Letting $a=(y^1)_{x^1}$, $b=(y^1)_{x^2}$,
and $c=(y^1)_{x^3}$, the previous system can be rewritten as
\begin{align}
a\tfrac{\partial y^2}{\partial x^2}
-b\tfrac{\partial y^2}{\partial x^1} & =\exp (r-B),\label{S4/1}\\
a\tfrac{\partial y^2}{\partial x^3}
-c\tfrac{\partial y^2}{\partial x^1} & =
\tfrac{\partial C}{\partial x^1}\exp (-B),\label{S4/2}\\
b\tfrac{\partial y^2}{\partial x^3}
-c\tfrac{\partial y^2}{\partial x^2} & =
\tfrac{\partial C}{\partial x^2}\exp (-B).
\label{S4/3}
\end{align}
The determinant of this system vanishes
and its compatibility condition reads
\begin{equation}
c\exp r-b\tfrac{\partial C}{\partial x^1}
+a\tfrac{\partial C}{\partial x^2}=0.
\label{compatcondition}
\end{equation}
If \eqref{compatcondition} holds and, for example, we assume
$c\neq 0$, then the system \eqref{S4/1}--\eqref{S4/3}
is equivalent to
\begin{align}
\tfrac{\partial y^2}{\partial x^1} &
=\tfrac{1}{c}\Bigl( a\tfrac{\partial y^2}
{\partial x^3}-\tfrac{\partial C}{\partial x^1}\exp (-B)\Bigr)
\label{S5/1}\\
\tfrac{\partial y^2}{\partial x^2} &
=\tfrac{1}{c}\Bigl( b\tfrac{\partial y^2}{\partial x^3}
-\tfrac{\partial C}{\partial x^2}\exp (-B)\Bigr)
\label{S5/2}
\end{align}
The integrability condition of this system is
\begin{equation}
\begin{array}
[c]{ll}
\!bc(y^2)_{x^1x^3}-ac(y^2)_{x^2x^3}
= & \left( bc_{x^1}-ac_{x^2}\right) (y^2)_{x^3}\\
& -\exp (-B)\\
& \cdot\left( c_{x^1}C_{x^2}
-c_{x^2}C_{x^1}+cB_{x^1}C_{x^2}
-cB_{x^2}C_{x^1}\right) .
\end{array}
\label{eq'1}
\end{equation}
Taking the derivative with respect to $x^3$
in the equation \eqref{S5/1} we obtain
\begin{align*}
bc(y^2)_{x^1x^3} &
=-\frac{bc_{x^3}}{c}\left( a(y^2)_{x^3}
-C_{x^1}\exp (-B)\right)
+a_{x^3}b(y^2)_{x^3}+ab(y^2)_{x^3x^3}\\
& \quad-bC_{x^1}\exp (-B)bC_{x^1x^3}\exp (B)
+bC_{x^1}B_{x^3}\exp (-B).
\end{align*}
Similarly, from the equation \eqref{S5/2} we obtain
\begin{align*}
ac(y^2)_{x^2x^3} &
=-\frac{ac_{x^3}}{c}\left( b(y^2)_{x^3}
-C_{x^2}\exp (-B)\right)
+ab_{x^3}(y^2)_{x^3}+ab(y^2)_{x^3x^3}\\
& \quad
+aC_{x^2}B_{x^3}\exp (-B).
\end{align*}
Subtracting these two equations, we have
\begin{align*}
bc(y^2)_{x^1x^3}-ac(y^2)_{x^2x^3} &
=\left( a_{x^3}b-ab_{x^3}\right) (y^2)_{x^3}\\
& \quad
+\frac{c_{x^3}}{c}\left( bC_{x^1}-aC_{x^2}\right)
\exp (-B)\\
& \quad
+\left( aC_{x^2x^3}-bC_{x^1x^3}\right)
\exp (-B)\\
& \quad
+\left( bC_{x^1}-aC_{x^2}\right) B_{x^3}\exp (-B).
\end{align*}
Comparing this equation to \eqref{eq'1}, we obtain
\begin{equation}
\begin{array}
[c]{ll}
0= & \frac{c_{x^3}}{c}\left( bC_{x^1}-aC_{x^2}\right)
+aC_{x^2x^3}-bC_{x^1x^3}\\
& +\left( bC_{x^1}-aC_{x^2}\right) B_{x^3}
+c_{x^1}C_{x^2}-c_{x^2}C_{x^1}\\
\multicolumn{1}{r}{} &
\multicolumn{1}{r}{+c\left( B_{x^1}C_{x^2}-B_{x^2}C_{x^1}\right) .}
\end{array}
\label{eq'2}
\end{equation}
Taking the derivative with respect to $x^3$ in \eqref{compatcondition}
and substituting the left hand side of this equation into \eqref{eq'2},
we obtain
\begin{equation}
Y(C)=\left( \exp (r)c\right) _{x^3},
\label{eq'4}
\end{equation}
where
\[
Y=\rho X,\quad \rho =\frac{c_{x^3}}{c}+1+y^1,
\quad
X=b\frac{\partial }{\partial x^1}
-a\frac{\partial }{\partial x^2}.
\]
Similarly, the equation \eqref{compatcondition}
can be written as
\begin{equation}
X(C)=\exp (r)c.\label{eq'5}
\end{equation}
To sum up, the unknown function $C=C(x^1,x^2,x^3)$ should verify
the two linear equations \eqref{eq'4} and \eqref{eq'5}
whose coefficients depend only on the function $u$.
The integrability condition of the system
\eqref{S5/1}--\eqref{S5/2}, which ensures the existence
of the function $y^2$, is thus reduced to the existence of $C$.

Moreover, the equations \eqref{eq'4} and \eqref{eq'5} admit
a common solution if and only if,
$(\exp (r)c)_{x^3} =Y(C)=\rho X(C)=\rho c$, or equivalently,
$r_{x^3}c+c_{x^3}=\rho c$. This equation is readily seen
to be an identity by simply substituting the expressions
for $c$, $c_{x^3}$, and $\rho$ in terms of $r$. Then
\[
\det \tfrac{\partial(y^1,y^2,y^3,y^4)}
{\partial(x^1,x^2,x^3,x^4)}
=\left\vert
\begin{array}
[c]{cccc}
r_{x^1x^3} & r_{x^2x^3} & r_{x^3x^3} & 0\\
(y^2)_{x^1} & (y^2)_{x^2} & (y^2)_{x^3} & 0\\
0 & 0 & 1 & 0\\
C_{x^1} & C_{x^2} & C_{x^3} & 1
\end{array}
\right\vert =\exp (r-x^3r_{x^3})\neq 0.
\]
\end{proof}
\section{Forms of type $3$\label{type3}}
\subsection{Generic normal form\label{generic_normal_forms}}
\begin{theorem}
\label{Type3}
If $\Omega \in \mathcal{G}$ is of type $3$, then there exists
a coordinate system $(x^i)_{i=1}^4$ centred at the origin
and a smooth function $f$ such that,
\begin{equation}
\Omega =\exp (f)\left( \exp \left( (1+x^1)x^3\right)
dx^1\wedge dx^2+dx^3\wedge dx^4\right) ,\quad f(0)=0.
\label{Omega'}
\end{equation}
\end{theorem}
\begin{proof}
As $\Omega $ is of type $3$, there exists a system
of coordinates centred at the origin $(t^i)_{i=1}^4$
such that $\omega _\Omega =dt^1+t^2dt^3$.
In this case, $\Omega ^\prime =\exp (-t^1)\Omega $ is
of type $2$ as $\omega _{\Omega ^\prime }
=\omega _\Omega -dt^1$, and we can apply Theorem
\ref{type2}.
\end{proof}
\subsection{A finer classification\label{FinerClassification}}
For every $2$-form $\Omega $ of type $3$ or $4$
on a manifold $M$ of dimension $4$, there exists
a unique $1$-form $\varphi _\Omega $ such that,
\begin{equation}
\varphi _\Omega \wedge\Omega
=\omega _\Omega \wedge d\omega _\Omega .
\label{varphi}
\end{equation}
The assignment $\Omega\mapsto\varphi _\Omega $ is functorial,
i.e., $\phi ^\ast (\varphi _\Omega ) =\varphi _{\phi ^\ast \Omega }$,
$\forall\phi \in\operatorname*{Diff}M$; cf.\ Lemma \ref{lemma_omega}.
If $\Omega \in\mathcal{G}$ is of type $1$ or $2$,
then $\varphi _\Omega =0$.

A $2$-form $\Omega\in\mathcal{G}$ of type $3$ is said to be
of type $3.k$, for $k=0,\dotsc,4$, if $\varphi _\Omega $
is of constant class equal to $k$ on a neighbourhood of the origin.

If a $2$-form $\Omega $ of type $3$ is given by the formula
\eqref{Omega'}, then
\begin{align}
\omega _\Omega & =df+(1+x^1)dx^3,\label{local_expression_1}\\
\varphi _\Omega & =-\frac{\partial F}{\partial x^4}dx^1
+\exp \left( -(1+x^1)x^3\right)
\frac{\partial F}{\partial x^2}dx^3,
\label{local_expression_2}\\
F & =\exp \left( -f\right) ,\;F(0)=1,\label{fF}
\end{align}
and as a computation shows, we have
\begin{proposition}
\label{finer_class}
A $2$-form $\Omega\in\mathcal{G}$ is of type $3.k$,
$0\leq k\leq 2$, if and only if there exist a smooth function
$F_k$ and a coordinate system $(y^i)_{i=1}^4$ centered
at the origin such that, $\Omega =(F_k)^{-1}
\left( \exp \left( (1+y^1)y^3\right) dy^1\wedge
dy^2+dy^3\wedge dy^4\right) $, $F_k(0)=1$ and
\begin{enumerate}
\item[\emph{(i)}] $(F_0)_{y^2}=(F_0)_{y^4}=0$,
\item[\emph{(ii)}] $F_{1}=1+y^4$,
\item[\emph{(iii)}] $(F_2)_{y^4}=0$.
\end{enumerate}
A $2$-form $\Omega =F^{-1}\left( \exp \left( (1+x^1)x^3\right)
dx^1\wedge dx^2+dx^3\wedge dx^4\right) \in\mathcal{G}$
is of type $3.3$ if and only if,
\begin{equation}
\left\{
\begin{array}
[c]{r}
\left( F_{x^2}F_{x^2x^4}
-F_{x^4}F_{x^2x^2}\right) ^2(0)
+\left( F_{x^2}F_{x^4x^4}
-F_{x^4}F_{x^2x^4}\right)^2(0)>0,\\
\\
F_{x^2x^2}F_{x^4x^4}-\left( F_{x^2x^4}\right) ^2=0,
\end{array}
\right.
\label{type_3.3}
\end{equation}
and it is of type $3.4$ if and only if
$F_{x^2x^2}(0)F_{x^4x^4}(0)
-\left( F_{x^2x^4}(0)\right) ^2\neq 0$.
\end{proposition}
\begin{example}
\label{Example3}
If $\Omega =\exp \left\{ x^4+(1+x^1)x^3\right\}
dx^1\wedge dx^2+\exp (\lambda)dx^3\wedge dx^4$, where
$\lambda\in C^\infty (x^1,\dotsc,x^4)$, then
$\omega={\frac{\partial\lambda}{\partial x^1}}dx^1
+{\frac{\partial\lambda }{\partial x^2}}dx^2+(1+x^1)dx^3+dx^4$.
If $\lambda _{x^2}=0$ and $\rho (0,0,0)\neq 0$, with
$\rho =1-\lambda _{x^1x^3}+(1+x^1)\lambda _{x^1x^4}$, then $\Omega $
is of class $3$ and $\varphi =\exp (-\lambda )\rho dx^1$.
Hence the class of $\varphi $ is $\leq 2$. The class of $\varphi $
is $1$ if and only if $\rho _{x^3}=\rho _{x^4}=0$. For example,
if $\lambda\in\mathbb{R}$, then
$\operatorname*{class}\varphi =1$; if $\lambda=x^1x^3x^4$, then
$\operatorname*{class}\varphi =2$.
\end{example}
\section{Invariant functions\label{inv1}}
Let $M$ be an arbitrary $C^\infty $-manifold and let
$\bar{\phi}\colon \wedge ^2T^\ast M\to \wedge ^2T^\ast M$
be the natural lift of a diffeomorphism
$\phi \in \operatorname*{Diff}M$; i.e.,
$\bar{\phi}(w)=(\phi ^{-1})^\ast w$
for every $2$-covector $w\in\wedge ^2T^\ast M$. If $\Omega $
is a $2$-form on $M$, then
$\bar{\phi }\circ \Omega \circ \phi ^{-1}
=(\phi ^{-1})^\ast \Omega $.
For every $r\geq 0$, let
\[
\begin{array}
[c]{l}
J^r\bar{\phi}
\colon J^r\bigl(\wedge ^2T^\ast M\bigr)
\to J^r\bigl(\wedge ^2T^\ast M\bigr)
\smallskip\\
J^r\bar{\phi }(j_x^r\Omega )
=j_{\phi (x)}^r(\bar{\phi }\circ \Omega \circ \phi ^{-1})
\end{array}
\]
be the $r$-jet prolongation of $\bar{\phi}$. A subset
$S\subseteq J^r(\wedge ^2T^\ast M)$ is said to be natural
if $\left( J^r\bar{\phi }\right) (S)\subseteq S$ for every
$\phi \in \operatorname*{Diff}M$. Let
$S\subseteq J^r(\wedge ^2T^\ast M)$ be a natural embedded submanifold.
A smooth function $I\colon S\rightarrow\mathbb{R}$ is said to be
invariant under diffeomorphisms or even
$\operatorname*{Diff}M$-invariant (cf.\ \cite{Kumpera}) if
\begin{equation}
I\circ J^r\bar{\phi}=I,\quad
\forall \phi \in \operatorname*{Diff}M.
\label{invariance1}
\end{equation}
If we set $\mathcal{I}(\Omega )=\mathcal{I}\circ j^r\Omega $,
for a given $2$-form $\Omega $ on $M$, then the invariance
condition \eqref{invariance1} reads
\begin{equation}
I\left( (\phi ^{-1})^\ast \Omega )(\phi (x)\right) =I(\Omega )(x),
\quad\forall x\in M,\;\forall\phi\in\operatorname*{Diff}M,
\label{invariance2}
\end{equation}
thus leading us to the naive definition of an invariant, as being a function
depending on the coefficients of $\Omega $ and its partial derivatives up to a
certain order, which remains unchanged under arbitrary changes of coordinates.
\begin{proposition}
Let $M$ be a smooth manifold of dimension $4$ and let $I$ be
the function defined by,
\[
\begin{array}
[c]{l}
I\colon J^2O_M\rightarrow\mathbb{R},
\smallskip\\
I(j_x^2\Omega )\omega _x\wedge\omega _x
=\left( d\omega _\Omega \right)_x
\wedge\left( d\omega _\Omega \right) _x.
\end{array}
\]
If $\omega _\Omega $ is of class $4$ (i.e., $\Omega $ is of type $4$,
according to \emph{\S \ref{intro/prelim}}) at $x$,
and $A_\Omega \colon TM\to TM$
is the endomorphism, $d\omega _\Omega (X,Y)=\Omega (A_\Omega
X,Y)=\Omega (X,A_\Omega Y)$, $\forall X,Y\in T_xM$, then
$\det \left( \lambda\mathrm{id}-\left( A_\Omega \right) _x\right)
=(\lambda ^2+I(j_x^2\Omega ))^2$.
\end{proposition}
\begin{proof}
Let $(x^1,\dotsc,x^4)$ be a system of coordinates centred at $x\in M$
such that, $\omega _\Omega =(1+x^1)dx^2+x^3dx^4$.
If $\Omega =\sum _{i<j}F_{ij}dx^i\wedge dx^j$,
then $F_{12}+F_{34}=0$, as differentiating \eqref{omega/m} it follows:
\[
0=d\omega _\Omega \wedge\Omega =\sum\nolimits_{i<j}F_{ij}
(dx^1\wedge dx^2+dx^3\wedge dx^4)\wedge dx^i\wedge dx^j.
\]
If we set $F_{ij}=-F_{ji}$ for $i\geq j$, and
$A_\Omega =A_{j}^idx^j\otimes\frac{\partial }{\partial x^i}$,
then $A_{j}^i=H_{ik}F^{kj}$,
with $H_{ij}=\delta_{1i}\delta_{2j}+\delta_{3i}\delta_{4j}$,
$1\leq i<j\leq 4$,
$H_{ij}+H_{ji}=0$, $i,j=1,\dotsc,4$, and
$\left( F^{ij}\right) _{i,j=1}^4$
is the inverse matrix of $\left( F_{ij}\right) _{i,j=1}^4$.
Hence
\[
\left( A_{j}^i\right) _{i,j=1}^4
=\det (F_{ij})^{-\frac{1}{2}}\left(
\begin{array}
[c]{cccc}
F_{34} & 0 & -F_{14} & F_{13}\\
0 & F_{34} & -F_{24} & F_{23}\\
F_{23} & -F_{13} & F_{12} & 0\\
F_{24} & -F_{14} & 0 & F_{12}
\end{array}
\right) ,
\]
where, $\det (F_{ij})^{\frac{1}{2}}
=\operatorname*{Pfaffian}(F_{ij})
=F_{12}F_{34}+F_{14}F_{23}-F_{13}F_{24}$,
and
\[
\det \left( \lambda\mathrm{id}-A_\Omega \right)
=\left( \lambda^2+\det (F_{ij})^{-\frac{1}{2}}\right) ^2,
\]
thus finishing the proof.
\end{proof}
\begin{remark}
The class of $\omega _\Omega $ is $\leq 3$ at $x\in M$
if and only if,
$I(j_x^2\Omega )=0$.
\end{remark}
\section{Forms of type $4$\label{type4}}
\subsection{Normal forms modulo $\operatorname{Aut}\omega _\Omega $}
Below, we use the standard notation about derivatives after a comma,
namely
$f_{\alpha,a}=\partial f_\alpha /\partial x^a$,
$f_{\alpha,ab}=\partial ^2f_\alpha /\partial x^a\partial x^b$, etc.

Let $\Omega =\sum _{i<j}F_{ij}dx^i\wedge dx^j$ be a germ
of differential $2$-form of type $4$ at $x_0\in M$ such that,
\begin{equation}
\omega _\Omega =\left( 1+x^1\right) dx^2+x^3dx^4,
\label{omega}
\end{equation}
where $(x^1,\dotsc,x^4)$ is a coordinate system centred at the origin.
By writing the equation \eqref{omega/m} in the system $(x^i)_{i=1}^4$,
where $\omega _\Omega $ is given as in \eqref{omega}, we conclude that
the former equation is equivalent to the following system:
\begin{equation}
\begin{array}
[c]{l}
F_{12,3}-F_{13,2}+F_{23,1}=-(1+x^1)F_{13},
\smallskip\\
F_{12,4}-F_{14,2}+F_{24,1}=x^3F_{12}-(1+x^1)F_{14},
\smallskip\\
F_{13,4}-F_{14,3}+F_{34,1}=x^3F_{13},
\smallskip\\
F_{23,4}-F_{24,3}+F_{34,2}=x^3F_{23}+(1+x^1)F_{34}.
\end{array}
\label{system1}
\end{equation}
Moreover, differentiating \eqref{omega/m} we obtain
$(dx^1\wedge dx^2+dx^3\wedge dx^4)\wedge\Omega =0$.
Hence $F_{34}=-F_{12}$ and we can use it to eliminate $F_{34}$
and its derivatives in \eqref{system1}, thus obtaining
the equivalent system,
\begin{equation}
\begin{array}
[c]{rl}
\text{$e$}_{1}\equiv & F_{13,2}-(1+x^1)F_{13}-F_{12,3}-F_{23,1}=0,
\smallskip\\
\text{$e$}_2\equiv & F_{12,4}-x^3F_{12}-F_{14,2}+(1+x^1)F_{14}
+F_{24,1}=0,
\smallskip\\
\text{$e$}_3\equiv & F_{13,4}-x^3F_{13}-F_{12,1}-F_{14,3}=0,
\smallskip\\
\text{$e$}_4\equiv & F_{12,2}-(1+x^1)F_{12}-F_{23,4}+x^3F_{23}
+F_{24,3}=0.
\end{array}
\label{system2}
\end{equation}
By performing the first prolongation of the system \eqref{system2}
a unique first-order constraint is obtained, namely,
\[
(1+x^1)F_{12,1}-x^3F_{12,3}+x^3F_{13,2}-(1+x^1)F_{13,4}
+(1+x^1)F_{14,3}-x^3F_{23,1}=0.
\]
\begin{proposition}
\label{proposition_system}
The general solution to the system
\emph{\eqref{system2}} is given by the following formulas:
\begin{equation}
F_{14}={\displaystyle\int_0^{x^3}}\left[ \left( X^4-x^3\right)
F_{13}-X^1F_{12}\right] dx^3+G_{14}\left( x^1,x^2,x^4\right)
,\label{F14}
\end{equation}
\begin{equation}
F_{23}={\displaystyle\int_0^{x^1}}\left[ \left( X^2-1-x^1\right)
F_{13}-X^3F_{12}\right] dx^1+G_{23}\left( x^2,x^3,x^4\right)
,\label{F23}
\end{equation}
\begin{equation}
F_{24}={\displaystyle\int_0^{x^1}}\left[ \left( x^3-X^4\right)
F_{12}+\left( X^2-1-x^1\right) F_{14}\right] dx^1+G_{24}
\left( x^2,x^3,x^4\right) ,
\label{F24}
\end{equation}
where $X^i=\frac{\partial }{\partial x^i}$, $1\leq i\leq 4$, and
\[
\begin{array}
[c]{l}
F_{12}(x^1,x^2,x^3,x^4),F_{13}(x^1,x^2,x^3,x^4),G_{14}
\left(
x^1,x^2,x^4\right) ,G_{23}\left( x^2,x^3,x^4\right) ,\\
\tilde{G}_{24}(x^2,x^4),
\end{array}
\]
are arbitrary smooth functions and $G_{24}$ is given by,
\[
G_{24}\left( x^2,x^3,x^4\right)
={\displaystyle\int_0^{x^3}}
\left( G_{23,4}-x^3G_{23}\right)
\left( x^2,x^3,x^4\right)
dx^3+\tilde{G}_{24}(x^2,x^4).
\]
\end{proposition}
\begin{proof}
From $e_{1}$, $e_2$, $e_3$ we obtain \eqref{F14}, \eqref{F23},
\eqref{F24}. Replacing \eqref{F14} into \eqref{F24}, the functions
$F_{14}$, $F_{23}$, and $F_{24}$ are written in terms of $F_{12}$
and $F_{13}$; the explicit expression for $F_{24}$ is
\begin{align*}
F_{24} & =G_{24}\left( x^2,x^3,x^4\right) +\int _0^{x^1}
\left( x^3-X^4\right) \left( F_{12}\right) dx^1\\
& +\int _0^{x^1}\left[ \left( X^2-1-x^1\right) \int _0^{x^3}
\left\{ \left( X^4-x^3\right) F_{13}-X^1F_{12}\right\}
dx^3\right] dx^1\\
& +\int_0^{x^1}\left( X^2-1-x^1\right) \left( G_{14}
\left( x^1,x^2,x^4\right) \right) dx^1.
\end{align*}
Replacing the expressions \eqref{F14}, \eqref{F23} and \eqref{F24}
into $e_1$, $e_2$, and $e_3$, these equations are seen to hold
identically.
Furthermore, replacing the aforementioned expressions into $e_4$,
after simplification, it follows:
\begin{multline*}
\int_0^{x^1}\left[ F_{12}-\left( X^2-1-x^1\right) X^1
F_{12}\right] dx^1+(x^3-X^4)G_{23}+X^3G_{24}\\
=\left( 1+x^1-X^2\right) F_{12},
\end{multline*}
and taking derivatives with respect to $x^1$ it follows:
$(x^3-X^4)G_{23}+X^3G_{24}=0$, and finally the expression
for $G_{24}$ in the statement is deduced.
\end{proof}
\subsection{The linear frame attached to $\Omega $}
\begin{proposition}
\label{ZandT}
Let $\Omega $ be a $2$-form $\Omega $ of rank $4$ on a manifold $M$
of dimension $4$, and let $Z_\Omega $, $T_\Omega $ be the vector fields
defined by $i_{Z_\Omega }\Omega =\omega _\Omega $,
$i_{T_\Omega }\Omega =\varphi _\Omega $,
then the following equations hold:
\begin{equation}
\begin{array}
[c]{rl}
\text{\emph{(i)}}
& \omega _\Omega \wedge d\Omega =0, \smallskip\\
\text{\emph{(iii)}}
& i_{Z_\Omega }d\Omega =0, \smallskip\\
\text{\emph{(v)}}
& \varphi _\Omega \wedge d\Omega =0, \smallskip\\
\text{\emph{(vii)}}
& \varphi _\Omega =-L_{Z_\Omega }\omega _\Omega ,
\smallskip\\
\text{\emph{(ix)}}
& i_{T_\Omega }d\omega _\Omega =I(\Omega )\omega _\Omega .
\end{array}
\begin{array}
[c]{rl}
\text{\emph{(ii)}} & d\omega _\Omega \wedge\Omega =0,
\smallskip\\
\text{\emph{(iv)}} & L_{Z_\Omega }\Omega =d\omega _\Omega ,
\smallskip\\
\text{\emph{(vi)}} & d\varphi _\Omega \wedge\Omega
=d\omega _\Omega \wedge d\omega _\Omega ,
\smallskip\\
\text{\emph{(viii)}} & \varphi _\Omega \wedge d\omega _\Omega
=-I(\Omega )\omega _\Omega \wedge\Omega,
\smallskip\\
&
\end{array}
\label{formulas}
\end{equation}
If $d\Omega\neq 0$, then
\begin{equation}
d\omega _\Omega \wedge d\varphi _\Omega
=-\tfrac{1}{2}Z_\Omega (I(\Omega ))\Omega \wedge \Omega .
\label{formula1}
\end{equation}
If $\omega _\Omega \wedge d\omega _\Omega \neq 0$, then
\begin{equation}
\Omega (Z_\Omega ,T_\Omega )=0,\label{Omega_Z_T}
\end{equation}
or equivalently, $\omega _\Omega (T_\Omega )
=\varphi _\Omega (Z_\Omega )=0$,
and the following equations hold:
\begin{equation}
\begin{array}
[c]{rl}
\text{\emph{(x)}} & \varphi _\Omega \wedge\omega _\Omega
=i_{T_\Omega }d\Omega,
\smallskip\\
\text{\emph{(xi)}} & L_{T_\Omega }\Omega
=d\varphi _\Omega +\varphi _\Omega \wedge\omega _\Omega ,
\smallskip\\
\text{\emph{(xii)}} & i_{[T_\Omega ,Z_\Omega ]}
\Omega =I(\Omega )\omega _\Omega -L_{Z_\Omega }\varphi _\Omega ,
\smallskip\\
\text{\emph{(xiii)}} & \left( i_{[T_\Omega ,Z_\Omega ]}\Omega\right)
\wedge\Omega =\left( d\varphi _\Omega -I(\Omega )\Omega\right)
\wedge \omega _\Omega ,
\smallskip\\
\text{\emph{(xiv)}} &
\omega _\Omega \left( \left[ T_\Omega ,Z_\Omega \right] \right) =0.
\end{array}
\label{formulas2}
\end{equation}
\end{proposition}
\begin{proof}
The formula (i) follows from the very definition of $\omega _\Omega $
in \eqref{omega/m}; differentiating this latter equation we obtain (ii).
From \eqref{omega/m} and the definition of $Z_\Omega $ we deduce
$i_{Z_\Omega }d\Omega =i_{Z_\Omega }(\omega _\Omega \wedge\Omega )=0$,
which is the formula (iii); hence
$L_{Z_\Omega }\Omega =di_{Z_\Omega }\Omega =d\omega _\Omega $,
which is (iv), and then
\[
\begin{array}
[c]{rll}
\varphi _\Omega \wedge d\Omega
= & \!\!\!\varphi _\Omega \wedge\omega _{\Omega }\wedge\Omega
& \text{[by virtue of \eqref{omega/m}]}\\
= & \!\!\!-\omega _\Omega \wedge\varphi _\Omega \wedge\Omega & \\
= & \!\!\!-\omega _\Omega \wedge\omega _\Omega \wedge d\omega _\Omega
& \text{[by virtue of \eqref{varphi}]}\\
= & \!\!\! 0, &
\end{array}
\]
thus proving (v). Taking (v) into account and differentiating
\eqref{varphi} we obtain (vi). Moreover, we have
\begin{align*}
\varphi _\Omega \wedge\Omega &
=i_{Z_\Omega }\Omega\wedge di_{Z_\Omega }\Omega
=i_{Z_\Omega }\Omega\wedge L_{Z_\Omega }\Omega =L_{Z_\Omega }
\left( i_{Z_\Omega }\Omega\wedge\Omega\right)
-L_{Z_\Omega } \left( i_{Z_\Omega }\Omega\right) \wedge\Omega \\
& =L_{Z_\Omega }d\Omega-\left( L_{Z_\Omega }\omega _\Omega \right)
\wedge\Omega.
\end{align*}
By virtue of (iv), $L_{Z_\Omega }d\Omega =dL_{Z_\Omega }\Omega
=d\left(
d\omega _\Omega \right) =0$;
hence $\left( \varphi _\Omega
+L_{Z_\Omega }\omega _\Omega \right) \wedge\Omega =0$,
and (vii) follows. Moreover, from (vii) we obtain
\[
\varphi _\Omega \wedge d\omega _\Omega
=-d\omega _\Omega \wedge L_{Z_\Omega }
\omega _\Omega =-d\omega _\Omega \wedge i_{Z_\Omega }d\omega _\Omega
=-\tfrac{1}{2}i_{Z_\Omega } \left( d\omega _\Omega \wedge d\omega _\Omega
\right)
\]
and (viii) follows from the very definitions of $I(\Omega )$ and $Z_\Omega $.
From (ii) and (viii), we have
$0=i_{T_\Omega }\left( d\omega _\Omega \wedge\Omega\right)
=i_{T_\Omega }d\omega _\Omega \wedge\Omega
+d\omega _\Omega \wedge\varphi _\Omega
=\left( i_{T_\Omega }d\omega _{\Omega
}-I(\Omega )\omega _\Omega \right) \wedge\Omega $. Hence (ix) follows.

Assume $d\Omega \neq 0$, and let $f$ be the function defined by, $d\omega
_\Omega \wedge d\varphi _\Omega =f\Omega \wedge\Omega $. By differentiating
(viii): $f\Omega\wedge\Omega =-d(I(\Omega ))\wedge d\Omega $, by virtue
of \eqref{omega/m} and (ii), and contracting with $Z_\Omega $ we obtain
$\left( 2f+Z_\Omega (I(\Omega ))\right) d\Omega =0$. Contracting (v) and
\eqref{omega/m} with $T_\Omega $, we obtain $\varphi _\Omega \wedge
i_{T_\Omega }d\Omega =0$,
$\Omega (Z_\Omega ,T_\Omega ) \Omega-\omega _\Omega \wedge\varphi _\Omega
=i_{T_\Omega }d\Omega$ , respectively.
Multiplying exteriorly the latter equation by $\varphi _\Omega $, we can
conclude that $\Omega (Z_\Omega ,T_\Omega )$ vanishes by virtue of the
assumption and taking the definition of $\varphi _\Omega $ into account. As
for the item (x), we have
\[
i_{T_\Omega }d\Omega =i_{T_\Omega }(\omega _\Omega \wedge\Omega )
=-\omega _\Omega \wedge i_{T_\Omega }\Omega =\varphi _\Omega \wedge
\omega _\Omega .
\]
Therefore, $L_{T_\Omega }\Omega
=di_{T_\Omega }\Omega+i_{T_\Omega }d\Omega
=d\varphi _\Omega +\varphi _\Omega \wedge\omega _\Omega $, thus proving
(xi). Moreover, (xii) is deduced as follows:
\[
\begin{array}
[c]{lll}
i_{[T_\Omega ,Z_\Omega ]}\Omega & \!\!
= i_{T_\Omega }L_{Z_\Omega }\Omega
-L_{Z_\Omega }i_{T_\Omega }\Omega
& \text{[by (iv)]}\\
& \!\! =i_{T_\Omega }d\omega _\Omega
-i_{Z_\Omega }di_{T_\Omega }\Omega
-d\left( i_{Z_\Omega }i_{T_\Omega }\Omega\right)
& \text{[as }
\Omega (Z_\Omega ,T_\Omega )=0\text{]}\\
& \!\! =i_{T_\Omega }d\omega _\Omega
-i_{Z_\Omega }d\varphi _\Omega & \\
& \!\! =I(\Omega )\omega _\Omega
-L_{Z_\Omega }\varphi _\Omega . & \text{[by (ix)]}
\end{array}
\]
From (xii) and (vi) we obtain
\begin{align*}
\left( i_{[T_\Omega ,Z_\Omega ]}\Omega \right) \wedge \Omega
& =\left(
I(\Omega )\omega _\Omega -i_{Z_\Omega }d\varphi _\Omega \right)
\wedge \Omega\\
& =I(\Omega )\omega _\Omega \wedge\Omega-i_{Z_\Omega }
\left( d\varphi _\Omega \wedge\Omega\right)
+d\varphi _\Omega \wedge\omega _\Omega \\
& =I(\Omega )\omega _\Omega \wedge\Omega-i_{Z_\Omega }
\left( I(\Omega )\Omega\wedge\Omega\right)
+d\varphi _\Omega \wedge\omega _\Omega \\
& =\left( d\varphi _\Omega -I(\Omega )\Omega \right)
\wedge\omega _\Omega ,
\end{align*}
thus proving (xiii). Finally, from (ix) we deduce
$i_{Z_\Omega }i_{T_{\Omega
}}d\omega _\Omega =0$, which is equivalent to (xiv) as we assume
$\omega _\Omega (T_\Omega )=0$.
\end{proof}
\begin{example}
\label{Example1}
If $\Omega =\exp (a)dx^1\wedge dx^2
+\exp (b)dx^3\wedge dx^4$, $a,b\in C^\infty (\mathbb{R}^4)$,
then
\begin{align*}
\tfrac{1}{2}\Omega\wedge\Omega
& =\exp (a+b)dx^1\wedge dx^2\wedge dx^3\wedge dx^4,\\
\omega _\Omega
& =b_{x^1}dx^1+b_{x^2}dx^2 +a_{x^3}dx^3+a_{x^4}dx^4,
\end{align*}
\[
d\omega _\Omega =c_{x^1x^3}dx^1\wedge dx^3+c_{x^1x^4}dx^1 \wedge
dx^4+c_{x^2x^3}dx^2\wedge dx^3+c_{x^2x^4}dx^2 \wedge dx^4,
\]
with $c=a-b$, and $I(\Omega )=\exp (-a-b)(c_{x^1x^4}c_{x^2x^3}
-c_{x^1x^3}c_{x^2x^4})$. Letting $a=x^3$, $b=0$ in the formulas
above, one has $I(\Omega )=0$. Similarly, letting $a=(1+x^1)x^3$, $b=0$,
one obtains $I(\Omega )=0$. This shows that the invariant $I$ does not
distinguish the types $1$ and $2$. Moreover, letting $u=c_{x^1}$,
$v=c_{x^2}$, the condition $I(\Omega )=0 $ is equivalent to saying that
$\frac{\partial(u,v)}{\partial(x^3,x^4)}=0$. Hence, $c_{x^2}=\phi
(x^1,x^2,c_{x^1})$, $\phi$ being an arbitrary smooth function on
$\mathbb{R}^3$.
\end{example}
\begin{theorem}
\label{TheBasis}
Let $Z_\Omega $, $T_\Omega $ be the vector fields attached
to a $2$-form $\Omega $ of type $4$ on $M$ according to
\emph{Proposition \ref{ZandT}}, and let $U_\Omega $
and $V_\Omega $ be the vector fields defined as follows:
$U_\Omega =\left[ Z_\Omega ,T_\Omega \right] $,
$V_\Omega =\left[ Z_\Omega ,U_\Omega \right]
=\left[ Z_\Omega ,
\left[ Z_\Omega ,T_\Omega \right] \right] $.
On a dense open subset $O_M^4\subset J^4O_M$,
the four vector fields $(Z_\Omega ,T_\Omega ,
U_\Omega ,V_\Omega )$ constitute a linear frame
for the tangent bundle of $M$. The function
$J(\Omega )$ defined by the formula
\begin{equation}
J(\Omega )
=\varphi _\Omega (U_\Omega )=\varphi _\Omega ([Z_\Omega ,T_\Omega ]),
\label{J}
\end{equation}
is a non-vanishing third-order differential invariant. Let
$(Z_{\Omega },T_\Omega ^\prime ,U_\Omega ^\prime ,V_\Omega ^\prime )$
be the linear frame defined by the following formulas:
\begin{align*}
T_\Omega ^\prime & =T-\tfrac{Z_\Omega (J(\Omega ))}{J(\Omega )}
Z_{\Omega },\\
U_\Omega ^\prime & =\tfrac{1}{J(\Omega )} U_\Omega
-\tfrac{\Omega (U_\Omega ,V_\Omega )}{J(\Omega )^2}Z_\Omega ,\\
V_\Omega ^\prime & =\tfrac{1}{J(\Omega )}V_\Omega .
\end{align*}
Two $2$-forms $\Omega $ and $\bar{\Omega }$, the $4$-jets
of which take values in $O_M^4$, are
$\operatorname*{Diff}M$-equivalent if and only if the linear frames
$(Z_\Omega ,T_\Omega ^\prime ,U_\Omega ^\prime ,V_{\Omega }^\prime )$
and $(Z_{\bar{\Omega }},T_{\bar{\Omega }}^\prime ,
U_{\bar{\Omega }}^\prime ,V_{\bar{\Omega }}^\prime )$
are $\operatorname*{Diff}M$-equivalent.
\end{theorem}
\begin{proof}
From the very definitions of the $1$-forms $\omega _\Omega $
and $\phi _\Omega $---in \eqref{omega/m} and \eqref{varphi},
respectively---it follows that
$\left( Z_\Omega \right) _x$ depends on $j_x^1\Omega $ and
$\left( T_\Omega \right) _x$ depends on $j_x^2\Omega $. Hence,
$\left( U_\Omega \right) _x$ depends on $j_x^3\Omega $ and
$\left( V_\Omega \right) _x$ depends on $j_x^4\Omega $.
Consequently the $\operatorname{Diff}M$-equivariant mapping
\[
\begin{array}
[c]{l}
\Upsilon \colon J^4O_M\to \wedge ^4TM,\\
\Upsilon \left( j_x^4\Omega\right)
=(Z_\Omega \wedge T_\Omega \wedge U_\Omega \wedge V_\Omega )_x,
\end{array}
\]
is well defined and, as a computation shows, once a point $x\in M$
has been fixed, the components of
$\Upsilon _x\colon J_x^4O_M\to \wedge _x^4TM$
are expressed as rational functions of the coordinate functions
$(y_{ij,I})$, $|I|\leq 4$, defined by $y_{ij,I}(j^4_x\Omega )
=\frac{\partial ^{|I|}(y_{ij}\circ \Omega )}{\partial x^I}(x)$.
In particular, the components of $\Upsilon _x$ are of class
$C^\Omega $. Hence, in order to prove that the open subset
$\Upsilon ^{-1}\left( \wedge ^4TM\backslash\{ 0\} \right) $
is dense it will suffice to prove that this subset is not empty.
In fact, if $\Omega =\sum _{i<j}F_{ij}dx^i\wedge dx^j$, where
\begin{equation}
\left.
\begin{array}
[c]{lll}
F_{12}=1+x^1,
& F_{13}=c\exp \left[ (1+x^1)x^2+x^3x^4\right] ,
& F_{14}=1, \medskip\\
F_{23}=-\lambda,
& F_{24}=\tfrac{1}{2}\lambda(x^3)^2-x^1
-\tfrac{1}{2}(x^1)^2,
& F_{34}=-F_{12},
\end{array}
\right\} \; c,\lambda \in \mathbb{R},
\label{example}
\end{equation}
then we have
$\omega _\Omega =(1+x^1)dx^2+x^3dx^4$,
$\varphi _{\Omega }=\varphi _idx^i$, with
\[
\begin{array}
[c]{llll}
\varphi _1=\frac{x^3F_{13}}{\lambda +F_{13}F_{24}},
& \varphi _2
=\frac{-\lambda x^3}{\lambda +F_{13}F_{24}},
& \varphi _3
=-\frac{(1+x^1)F_{13}}{\lambda +F_{13}F_{24}},
& \varphi _4
=-\frac{1+x^1}{\lambda +F_{13}F_{24}}.
\end{array}
\]
As the assignment
$\Omega \mapsto(Z_\Omega ,T_\Omega ,U_\Omega ,V_\Omega )$
is $\operatorname{Diff}M$-equivariant, we can confine ourselves
to study the rank of the system
$(Z_\Omega ,T_\Omega ,U_\Omega ,V_\Omega )$ at the point $x_0$.
If $\left( Z_\Omega \right) _{x_0}
=z^i\frac{{\small \partial}}{{\small \partial x}^i}|_{x_0}$,
$\left( T_{\Omega }\right) _{x_0}
=t^i\frac{{\small \partial}}{{\small \partial x}^i}|_{x_0}$,
$\left( U_\Omega \right) _{x_0}
=u^i\frac{{\small \partial }}{{\small \partial x}^i}|_{x_0}$,
$\left( V_\Omega \right) _{x_0}
=v^i\frac{{\small \partial}}{{\small \partial x}^i}|_{x_0}$,
then
\[
\begin{array}
[c]{llll}
z^1=0, & z^2=-\frac{\lambda }{\lambda +1}, & z^3=0,
& z^4=-\frac{1}{\lambda +1},
\end{array}
\]
\[
\begin{array}
[c]{llll}
t^1=-\frac{c}{\left( \lambda+1\right) ^2}, & t^2
=\frac{\lambda -1}{\left( \lambda+1\right) ^2}, & t^3
=\frac{c\lambda}
{\left( \lambda+1\right) ^2}, & t^4
=\frac{c^2+2}{\left( \lambda+1\right) ^2},
\end{array}
\]
\[
\begin{array}
[c]{llll}
u^1=\frac{c}{(\lambda+1)^2}, & u^2=-\frac{c\lambda(2\lambda+c)}
{(\lambda+1)^4}, & u^3=-\frac{c\lambda^2}{(\lambda+1)^3}, &
u^4=-\frac{ c\left[ (2c-1)\lambda^2+(2c+1)\lambda+c \right] }
{(\lambda+1)^4},
\end{array}
\]
\[
\begin{array}
[c]{ll}
v^1=-\frac{c}{\left( \lambda+1\right) ^2}, & v^2
=\frac{c\lambda
\left[ 4\lambda^2+(3c+1)\lambda+c-1 \right] }
{\left( \lambda+1\right) ^5},\\
v^3=\frac{c\lambda^3}{\left( \lambda+1\right) ^4}, & v^4
=\frac{c\left[ 2(2c-1)\lambda ^3+2(2c+1)\lambda ^2+(3c+2)\lambda +c\right] }
{\left( \lambda+1\right) ^5}.
\end{array}
\]

The determinant of $\left( \left( Z_\Omega \right) _{x_0},
\left( T_\Omega \right) _{x_0},\left( U_\Omega \right) _{x_0},
\left( V_\Omega \right) _{x_0}\right) $ in the basis
$\left( \frac{\partial }{\partial x^i}|_{x_0}\right) _{i=1}^4$
is equal to
\[
-c^2\lambda^2\frac{\left( c^2+1\right) \lambda^2+\left(
c-c^2+2\right) \lambda+c+1}{(\lambda+1)^{10}}.
\]
Hence it suffices to take $c\neq 0$, $\lambda\neq 0,-1$,
$\lambda\neq\frac{c^2-c-2\pm c\sqrt{c^2-6c-7}}{2(c^2+1)}$.

Moreover, if there exists $\phi \in \operatorname{Diff}M$
such that
$\phi ^\ast \Omega =\bar{\Omega }$, then
$\phi ^\ast \omega _\Omega =\omega _{\bar{\Omega }}$,
$\phi ^\ast \varphi _\Omega =\varphi _{\bar{\Omega }}$.
Hence
$\phi ^{-1}\cdot Z_\Omega =Z_{\bar{\Omega }}$,
$\phi ^{-1}\cdot T_\Omega =T_{\bar{\Omega }}$;
from the very definitions of $U_\Omega $ and $V_\Omega $
it thus follows:
$\phi ^{-1}\cdot U_\Omega =U_{\bar{\Omega }}$,
$\phi ^{-1}\cdot V_{\Omega }=V_{\bar{\Omega }}$.
Next, we compute the functions
$\Omega (Z_{\Omega },T_\Omega )$, $\Omega (Z_\Omega ,U_\Omega )$,
$\Omega (Z_\Omega ,V_{\Omega })$, $\Omega (T_\Omega ,U_\Omega )$,
$\Omega (T_\Omega ,V_\Omega )$. To do this, we use the properties
(i)-(xiv) of Proposition \ref{ZandT}. From
\eqref{Omega_Z_T} it follows: $\Omega (Z_\Omega ,T_\Omega )=0$.
Moreover, we have
\begin{equation}
\Omega (Z_\Omega ,U_\Omega )
=\Omega (Z_\Omega ,[Z_\Omega ,T_\Omega ])
=\omega _\Omega ([Z_\Omega ,T_\Omega ])\overset{\text{(xiv)}}{=}0,
\label{Omega_Z_U}
\end{equation}
\begin{equation}
\begin{array}
[c]{rll}
\Omega (Z_\Omega ,V_\Omega )
& =\Omega (Z_\Omega ,[Z_\Omega ,[Z_{\Omega },T_\Omega ]])
& \\
& =\omega _\Omega ([Z_\Omega ,[Z_\Omega ,T_\Omega ]])
& \\
& =i_{[Z_\Omega ,[Z_\Omega ,T_\Omega ]]}\omega _\Omega
& \\
& =L_{Z_\Omega }
\left( i_{[Z_\Omega ,T_\Omega ]}\omega _\Omega \right)
-i_{[Z_\Omega ,T_\Omega ]}
\left( L_{Z_\Omega }\omega _\Omega \right)
& \\
& =-i_{[Z_\Omega ,T_\Omega ]}(L_{Z_\Omega }\omega _\Omega )
& \text{by (xiv)} \\
& =i_{[Z_\Omega ,T_\Omega ]}\varphi _\Omega
& \text{by (vii)}\\
& =J(\Omega ).
&
\end{array}
\label{Omega_Z_V}
\end{equation}
In addition,
\begin{equation}
\begin{array}
[c]{rl}
\Omega (T_\Omega ,U_\Omega )
& =\Omega (T_\Omega ,[Z_\Omega ,T_\Omega ])\\
& =\varphi _\Omega ([Z_\Omega ,T_\Omega ])\\
& =J(\Omega ),
\end{array}
\label{Omega_T_U}
\end{equation}
and furthermore we obtain
\begin{align*}
\Omega (T_\Omega ,V_\Omega )
&
=\Omega (T_\Omega ,[Z_\Omega ,[Z_{\Omega },T_\Omega ]])
\smallskip\\
& =\varphi _\Omega ([Z_\Omega ,[Z_\Omega ,T_\Omega ]])
\smallskip\\
& =i_{[Z_\Omega ,[Z_\Omega ,T_\Omega ]]}\varphi _\Omega
\smallskip\\
& =L_{Z_\Omega }(i_{[Z_\Omega ,T_\Omega ]}\varphi _\Omega )
-i_{[Z_\Omega ,T_\Omega ]}(L_{Z_\Omega }\varphi _\Omega ),
\end{align*}
and from (xii) we deduce
$L_{Z_\Omega }\varphi _\Omega =I(\Omega )\omega _\Omega
+i_{[Z_\Omega ,T_\Omega ]}\Omega $, and again from (xiv)
we obtain
\begin{align*}
i_{[Z_\Omega ,T_\Omega ]}(L_{Z_\Omega }\varphi _\Omega )
& =I(\Omega )\omega _\Omega \left( [Z_\Omega ,T_\Omega ]
+\Omega ([Z_\Omega ,T_{\Omega
}],[Z_\Omega ,T_\Omega ] \right) \\
& =0.
\end{align*}
Consequently,
\begin{equation}
\Omega (T_\Omega ,V_\Omega )=Z_\Omega (J(\Omega )).
\label{Omega_T_V}
\end{equation}
Therefore, the matrix of $\Omega $ in the basis
$(Z_\Omega ,T_{\Omega },U_\Omega ,V_\Omega )$ is
\[
\Lambda(\Omega )=\left(
\begin{array}
[c]{cccc}
0 & 0 & 0 & J(\Omega )\\
0 & 0 & J(\Omega ) & Z_\Omega J(\Omega )\\
0 & -J(\Omega ) & 0 & \Omega\left( U_\Omega ,V_\Omega \right) \\
-J(\Omega ) & -Z_\Omega J(\Omega ) &
-\Omega\left( U_\Omega ,V_{\Omega }\right) & 0
\end{array}
\right) .
\]
Hence
$\left( \Omega\wedge\Omega\right)
\left( Z_\Omega ,T_{\Omega },U_\Omega ,V_\Omega \right)
=2J(\Omega )^2$. As $\Omega $ is of rank $4$,
it follows that $J(\Omega )$ cannot vanish,
and the definition of the modified basis
$(Z_\Omega ,T_\Omega ^\prime ,
U_\Omega ^\prime ,V_\Omega ^{\prime })$ makes sense.

It is readily seen that the functions $J(\Omega )$,
$Z_\Omega (J(\Omega ))$,
$\Omega\left( U_\Omega ,V_\Omega \right) $
are differential invariants. Hence the linear frame
$(Z_\Omega ,T_\Omega ^\prime ,U_\Omega ^\prime ,V_\Omega ^\prime )$
also depends functorially on $\Omega $. In addition, as
a calculation shows, from the formulas \eqref{Omega_Z_T},
\eqref{Omega_Z_U}--\eqref{Omega_T_V}, we have
\[
\begin{array}
[c]{lll}
\Omega \left( Z_\Omega ,T_\Omega ^\prime \right) =0,
& \Omega \left( Z_\Omega ,U_\Omega ^\prime \right) =0,
& \Omega \left( Z_{\Omega },V_\Omega ^\prime \right) =1,\\
\Omega \left( T_\Omega ^\prime ,U_\Omega ^\prime \right) =1, &
\Omega \left( T_\Omega ^\prime ,V_\Omega ^\prime \right) =0, &
\Omega \left( U_\Omega ^\prime ,V_\Omega ^\prime \right) =0.
\end{array}
\]
If
$\phi\cdot(Z_\Omega ,T_\Omega ^\prime ,
U_\Omega ^\prime ,V_\Omega ^\prime )
=(Z_{\bar{\Omega }},T_{\bar{\Omega }}^\prime ,
U_{\bar{\Omega }}^\prime ,
V_{\bar{\Omega }}^\prime )$, i.e.,
$\phi ^{-1}\cdot Z_\Omega =Z_{\bar{\Omega }}$,
$\phi ^{-1}\cdot T_\Omega ^\prime =T_{\bar{\Omega }}^\prime $,
$\phi ^{-1}\cdot U_\Omega ^\prime =U_{\bar{\Omega }}^\prime $,
$\phi ^{-1}\cdot V_\Omega ^\prime =V_{\bar{\Omega }}^\prime $,
then the forms $\phi ^\ast \Omega $ and $\bar{\Omega }$
take the same value on the pairs
$(Z_{\bar{\Omega }},T_{\bar{\Omega }}^\prime )$,
$(Z_{\bar{\Omega }},U_{\bar{\Omega }}^\prime )$,
$(Z_{\bar{\Omega }},V_{\bar{\Omega }}^\prime )$,
$(T_{\bar{\Omega }}^\prime ,U_{\bar{\Omega }}^\prime )$,
$(T_{\bar{\Omega }}^\prime ,V_{\bar{\Omega }}^\prime )$,
$(U_{\bar{\Omega }}^\prime ,V_{\bar{\Omega }}^\prime )$,
as
$(\phi ^\ast \Omega )(Z_{\bar{\Omega }},T_{\bar{\Omega }}^\prime )
=(\phi ^\ast \Omega )(\phi ^{-1}\cdot Z_\Omega ,
\phi ^{-1}\cdot T_\Omega ^\prime )
=\Omega \left( Z_\Omega ,T_{\Omega }^\prime \right) \circ \phi =0$,
and similarly for the remaining pairs. Hence
$\bar{\Omega }=\phi ^\ast \Omega $, thus concluding the proof.
\end{proof}
\begin{corollary}
Let $M$ be a $4$-dimensional real analytic manifold,
let $\Omega $ and $\bar{\Omega }$ be two $2$-forms
of type $4$ and class $C^\omega $ the $4$-jets
of which belong to $O_M^4$, and let
\begin{eqnarray*}
(X_1,X_2,X_3,X_4) & = &(Z_\Omega ,T_\Omega ^\prime ,
U_\Omega ^\prime ,V_\Omega ^\prime ), \\
(\bar{X}_1,\bar{X}_2,\bar{X}_3,\bar{X}_4) &=&
(Z_{\bar{\Omega }},T_{\bar{\Omega }}^\prime
,U_{\bar{\Omega }}^\prime ,V_{\bar{\Omega }}^\prime ),
\end{eqnarray*}
be their associated linear frames respectively as in
\emph{Theorem \ref{TheBasis}}. Let $\nabla $, $\bar{\nabla }$
be the linear connections parallelizing these linear frames
respectively; i.e.,
$\nabla _{X_i}X_j=0$, $\bar{\nabla }_{\bar{X}_i}\bar{X}_j=0$,
$i,j=1,\dotsc,4$. Finally, let $T$, $\bar{T}$
be the torsion tensors of $\nabla $, $\bar{\nabla }$,
respectively. The germs of $\Omega $ and $\bar{\Omega }$
at $x_0$ and $\bar{x}_0$ are equivalent if and only if
there exists a linear isomorphism $A\colon T_{x_0}M
\to T_{\bar{x}_0}M$ mapping the tensor $(\nabla ^rT)_{x_0}$
into $(\bar{\nabla }^r\bar{T})_{\bar{x}_0}$
for every $r\in \mathbb{N}$.
\end{corollary}
\begin{proof}
As the linear frames $(X_1,\dotsc,X_4)$,
$(\bar{X}_1,\dotsc,\bar{X}_4)$ are of class $C^\omega $,
according to \cite[Corollary (5.6)]{GM},
they are equivalent if and only if the map $A$ exists,
and we can conclude by virtue of Theorem \ref{TheBasis}.
\end{proof}
\begin{remark}
As $\nabla $ and $\bar{\nabla }$ are flat connections,
the statement of the previous corollary is a particular case
of the equivalence problem for linear connections,
such as is formulated, for example, in
\cite[VI. Theorem 7.2]{KN}; but there is an alternative
(and equivalent) formulation of this corollary that better
shows the computational complexity of the $2$-form equivalence
problem. For every integer $r\geq 0$, let
$I^i_{j_1,\dotsc,j_r,k,l}$ be the invariant functions
defined by the following formula:
\[
\left( \nabla ^rT\right)
\left(
(X_{j_1})_x,\dotsc,(X_{j_r})_x,(X_k)_x,(X_l)_x
\right)
=\sum\nolimits_iI^i_{j_1,\dotsc,j_r,k,l}
\left( j^{r+4}_x\Omega \right) (X_i)_x,
\]
for all $i,j_1,\dotsc,j_r,k,l=1,\dotsc,4$ and every
$2$-form $\Omega $ of type $4$ taking values in $O_M^4$.
Then, the germs $\Omega $ and $\bar{\Omega }$ at $x\in M$
are equivalent if and only if all the following equations hold:
\[
I^i_{j_1,\dotsc,j_r,k,l}\left( j^{r+4}_x\Omega \right)
=I^i_{j_1,\dotsc,j_r,k,l}\left( j^{r+4}_x\bar{\Omega }\right).
\]
\end{remark}

\noindent\textbf{Authors' addresses}

\smallskip

\noindent(J.M.M.) \textsc{Instituto de Tecnolog\'{\i}as
F\'{\i}sicas y de la Informaci\'on, CSIC, C/ Serrano 144,
28006-Madrid, Spain.}

\noindent\emph{E-mail:\/} \texttt{jaime@iec.csic.es}

\medskip

\noindent(L.M.P.C.) \textsc{Departamento de Matem\'atica
Aplicada a las T.I.C., E.T.S.I. Sistemas Inform\'aticos,
Universidad Polit\'ecnica de Madrid,
Carrete\-ra de Valencia, Km.\ 7, 28031-Madrid, Spain.}

\noindent\emph{E-mail:\/} \texttt{luispozo@etsisi.upm.es}
\end{document}